\documentclass[11pt]{article}

\usepackage[T1]{fontenc}
\usepackage{lmodern}
\usepackage{color}
\usepackage{microtype}
\usepackage[margin=1in]{geometry}
\usepackage{amsmath,amssymb,amsthm,mathtools}
\usepackage[numbers,sort&compress]{natbib}
\usepackage[hidelinks]{hyperref}

\mathtoolsset{showonlyrefs}

\newtheorem{theorem}{Theorem}
\newtheorem{proposition}[theorem]{Proposition}
\newtheorem{lemma}[theorem]{Lemma}
\newtheorem{corollary}[theorem]{Corollary}
\theoremstyle{remark}
\newtheorem{remark}[theorem]{Remark}

\newcommand{\R}{\mathbb{R}}
\newcommand{\E}{\mathbb{E}}
\newcommand{\Pp}{\mathbb{P}}
\newcommand{\dd}{\,\mathrm{d}}
\newcommand{\ind}{\mathbf{1}}
\newcommand{\ascale}{c_0}
\newcommand{\resmgf}{L_T}
\newcommand{\curv}{\kappa}
\newcommand{\curvg}{v}
\newcommand{\tg}{\tilde{g}}
\newcommand{\gpt}{\texttt{GPT-5.6 Sol}}
\DeclareMathOperator{\Ai}{Ai}
\DeclareMathOperator{\Exp}{Exp}
\DeclareMathOperator{\Var}{Var}
\DeclareMathOperator*{\argmax}{arg\,max}

\title{Chernoff's Density Is Strongly Log-Concave}
\author{Xianyang Zhang and Quan Zhou \medskip   \\ 
Department of Statistics, Texas A\&M University}
\date{\today}

\begin{document}
\maketitle

\begin{abstract}
Let $f$ be the density of the Chernoff random variable $\argmax_{t\in\mathbb{R}}\{W(t)-t^2\}$, where $W$ is a two-sided Brownian motion.  
This note proves the conjecture of Balabdaoui and Wellner (2014) that $f$ is strongly log-concave. The proof  was generated in its entirety  by \gpt{}. 
\end{abstract}

\section{Introduction}
Let $W$ be a two-sided standard Brownian motion with $W(0)=0$, and let 
\begin{equation}
    Z=\argmax_{t\in\R}\{W(t)-t^2\}.
\end{equation} 
The density $f$ of $Z$ is sometimes known as Chernoff's density, which first arose in Chernoff's study of mode estimation~\citep{Chernoff1964} and has since become a canonical non-Gaussian limit law in asymptotic theory. In particular, it governs the pointwise limiting behavior of estimators in isotonic regression, monotone density and hazard estimation, interval censoring, and several other shape-constrained problems; see, for example, the applications discussed  in~\cite{GroeneboomWellner2001}. Balabdaoui and Wellner~\citep{BalabdaouiWellner2014} proved that   \(f\) is log-concave and conjectured that $f$ is strongly log-concave, i.e.,  
\[
\inf_{x\in\mathbb R}-(\log f)''(x)>0.
\]
Strong log-concavity places a distribution in a substantially more rigid class than ordinary log-concavity, as a uniform lower bound on the curvature of \(-\log f\) yields Gaussian-type concentration and associated functional inequalities. 
See~\citet{saumard2014log} for a comprehensive review on the properties of log-concave and strongly log-concave distributions.

We introduce some necessary notation. \citet{Chernoff1964} obtained the following product factorization of the density $f$: 
\begin{equation}\label{eq:factorization-wellner}
    f(x) = \frac{1}{2} \tg(x) \tg(-x). 
\end{equation} 
Groeneboom, in his work on the Brownian motion with parabolic drift~\citep{Groeneboom1989}, showed that the bilateral Laplace transform of  the one-sided factor $\tg$ is 
\begin{equation}\label{eq:laplace-wellner}
   \int_{\R} e^{sx} \tg (x)\dd x
      =\frac{2^{1/3} }{\Ai( 2^{-1/3} s)}, \qquad \Re(s) > -2^{1/3} a_1,  
\end{equation} 
where $\Ai$ is the Airy function, and $-a_1$ is the largest zero of $\Ai$. 
Multiplication of $\tg $ by a positive constant does not affect its logarithmic derivatives.  We therefore normalize it to a probability density $g$. Using~\eqref{eq:laplace-wellner} with $s = 0$ and letting $ \ascale:=2^{-1/3}$,  we obtain that 
\begin{equation}\label{eq:M-def}
    M(s):=\int_{\R}e^{sx}g(x)\dd x
      =\frac{\Ai(0)}{\Ai(\ascale s)},
    \qquad  \Re(s) > -\lambda_1, 
\end{equation}
where $\lambda_1 := a_1 /c_0$. Hence, we can rewrite $f$ as 
\begin{equation}\label{eq:factorization}
    f(x)=C_f \, g(x)g(-x), \qquad C_f = \frac{\ascale}{\Ai (0)^2}. 
\end{equation}  
Set
\begin{equation}\label{eq:v-kappa}
    \curvg(x):=-(\log g)''(x),
    \qquad
    \curv(x):=-(\log f)''(x)=\curvg(x)+\curvg(-x).
\end{equation} 
We are now ready to state our main result.  

\begin{theorem}\label{thm:main}
For the density function $g$ in \eqref{eq:M-def} and $\curv, \curvg$ in~\eqref{eq:v-kappa}, 
\begin{equation}\label{eq:v-positive-main}
    \curvg(x)>0, \qquad x\in\R, 
\end{equation}
and
\begin{equation}\label{eq:v-linear-main}
    \curvg(x)\ge 4x, \qquad x\in\R.
\end{equation}
Consequently,
\begin{equation}\label{eq:kappa-linear-main}
    \curv(x)>4|x|, \qquad x\in\R, 
\end{equation}
and
\begin{equation}\label{eq:uniform-main}
     \inf_{x\in\R}\curv(x)>0.
\end{equation}
Thus Chernoff's density is strongly log-concave. 
\end{theorem}

The proof of Theorem~\ref{thm:main} has two key ingredients, built on the Airy random series representation used by \citet{BalabdaouiWellner2014} and recorded in Section~\ref{sec:series}. First, in Section~\ref{sec:peeling}, we prove a general exponential peeling lemma and apply it to obtain the strict pointwise bound $\curvg(x)>0$. 
Second, in Section~\ref{sec:identity}, we use an Airy convolution identity of Menon and Srinivasan~\citep{MenonSrinivasan2010} to obtain the linear bound $\curvg(x)\ge 4x$. Since $f(x)\propto g(x)g(-x)$, these complementary bounds yield $\curv(x)>4|x|$; continuity and a compactness argument then imply that $\inf_{x\in\R}\curv(x)>0$.

We note that Balabdaoui and Wellner~\citep{BalabdaouiWellner2014} already provided numerical evidence for the strong log-concavity of $f$, though they did not obtain a proof. The purpose of this note is to give a fully analytic proof,  establishing in particular that $\curv $ is strictly positive everywhere and that $\curv$ is bounded away from zero in the tails.

\paragraph{AI Usage.} 
The whole proof was obtained by \gpt{} through a sequence of interactions. Its initial attempt (runtime: 110 min) contained a simple but fatal algebraic mistake. On a second attempt (runtime: 167 min), \gpt{} produced a correct proof, although the results in Section~\ref{sec:identity} were established by a substantially more involved argument. In a subsequent attempt (runtime: 134 min), \gpt{} identified the present proof strategy and independently derived the key identity used in Section~\ref{sec:identity}. Further discussions then led \gpt{} to recognize that an essentially equivalent identity had already been established by \citet{MenonSrinivasan2010}, and to streamline several parts of the argument accordingly. 
The rest of this work was written by \gpt{}; the authors carefully checked the argument, revised the exposition and added the comments.  

\section{The reciprocal-Airy random series}\label{sec:series}

Let $-a_1,-a_2,\ldots$ be the zeros of the Airy function, indexed by
\begin{equation}
    0<a_1<a_2<\cdots,
\end{equation}
and define 
\begin{equation}\label{eq:rates-means}
    \lambda_j:=\frac{a_j}{\ascale}=2^{1/3}a_j,
    \qquad
    b_j:=\frac1{\lambda_j}=\frac{\ascale}{a_j}.
\end{equation}
The standard asymptotic result about the zeros of $\Ai$ gives $b_j\asymp j^{-2/3}$~\citep{Olver1954}, and thus
\begin{equation}\label{eq:b-sums}
    \sum_{j=1}^\infty b_j^2<\infty,
    \qquad
    \sum_{j=1}^\infty b_j=\infty.
\end{equation}
The Hadamard product representation of $\Ai$~\citep{merkes1997univalence}, in the form used in
\cite[Proposition~2.2]{BalabdaouiWellner2014}, gives a constant $d\in\R$ such that
\begin{equation}\label{eq:M-product}
    M(s)= \int_{\R} e^{sx} g(x) \dd x = e^{ds}\prod_{j=1}^\infty
          \frac{e^{b_js}}{1+b_js}, \qquad \Re(s) > -\lambda_1. 
\end{equation}
The product converges locally uniformly away from its poles because $\sum_jb_j^2<\infty$ and
\begin{equation}\label{eq:taylor-Ms}
    \log \frac{e^{b_js}}{1+b_js} = \frac{1}{2} s^2 b_j^2 + O(s^3 b_j^3), \quad  \text{ as }  j \rightarrow \infty. 
\end{equation}

Let $(X_j)_{j \geq 1}$ be independent exponential random variables with rates $\lambda_j$  so that $\E X_j=b_j$.  Then, 
\begin{equation}
    \E e^{s(b_j-X_j)}=\frac{e^{b_js}}{1+b_js}, 
\end{equation}
for every real $s > -\lambda_1$. Since
\begin{equation}
    \sum_{j=1}^\infty\Var(b_j-X_j)=\sum_{j=1}^\infty b_j^2<\infty,
\end{equation}
the centered series
\begin{equation}\label{eq:Y-series}
    Y:=d+\sum_{j=1}^\infty(b_j-X_j)
\end{equation}
satisfies Kolmogorov's convergence criterion and thus converges almost surely and in $L^2$.  
For $s > -\lambda_1$, 
\begin{equation}
    \left\{ e^{\sum_{j=1}^N s (b_j  - X_j)} \colon N \geq 1 \right\}
\end{equation}
are uniformly integrable by a standard moment bound argument. Hence, the moment-generating function of $Y$ is given by~\eqref{eq:M-product}, which implies that $Y$ has density $g$. The random series representation~\eqref{eq:Y-series} was also used by Balabdaoui and Wellner~\citep[Section~3]{BalabdaouiWellner2014} in their partial argument for the conjectured strong log-concavity.

\section{Exponential peeling and strict curvature}\label{sec:peeling}

We first isolate the key probabilistic tool used in our proof in a general lemma, which combines a standard Esscher transform with the classical implication from log-concave densities to increasing hazard rates. The resulting curvature identity~\eqref{eq:general-hazard-curvature} is elementary, but its application to the random series $Y$ upgrades the previously known inequality $\curvg \geq 0$ to the pointwise strict inequality $\curvg > 0$.  

\begin{lemma}[Exponential peeling]\label{lem:peeling}
Let $E\sim\Exp(\lambda)$ have rate $\lambda>0$, let $T$ be independent of
$E$, and let $c\in\R$.  Assume that $T$ has a strictly positive $C^1$
log-concave density $q$ on $\R$ and that
\begin{equation}\label{eq:general-mgf-assumption}
    \resmgf(s):=\E e^{sT}<\infty
    \qquad\text{for every }s\in[-\lambda,\infty).
\end{equation}
If $G$ is the density of $c-E+T$, then $G$ is $C^2$ and
\begin{equation}\label{eq:general-strict-curvature}
    -(\log G)''(x)>0, \qquad x\in\R.
\end{equation}
More precisely, define the Esscher-transformed probability density
\begin{equation}\label{eq:general-tilt}
    p(t):=\frac{e^{-\lambda t}q(t)}{\resmgf(-\lambda)},
\end{equation}
let $S(z)=\int_z^\infty p(t)\dd t$, and let $h=p/S$ be its right hazard
rate.  Then
\begin{equation}\label{eq:general-survival-representation}
    G(x)=\lambda e^{\lambda(x-c)}\resmgf(-\lambda)S(x-c), 
\end{equation}
and
\begin{equation}\label{eq:general-hazard-curvature}
    -(\log G)''(x)=h'(x-c)>0.
\end{equation}
\end{lemma}

\begin{proof}
Conditioning on $T$ gives
\begin{align}
    G(x)  &= \int_{x - c}^\infty \lambda e^{-\lambda (c + t - x)} q(t) \dd t  \\ 
      &=\lambda e^{\lambda(x-c)}
        \E\left[e^{-\lambda T}\ind_{\{T\ge x-c\}}\right],
\end{align}
which is \eqref{eq:general-survival-representation}.  The density $p$ is
strictly positive, $C^1$, and log-concave because it is obtained from $q$ by
multiplication by a log-affine function.  Moreover, for every $\theta>0$, 
\begin{equation}\label{eq:tilted-positive-moments}
    \int_{\R}e^{\theta t}p(t)\dd t 
     =  \int_{\R} \frac{ e^{(\theta-\lambda) t} q(t)}{\resmgf(-\lambda)} \dd t   
      =\frac{\resmgf(\theta-\lambda)}{\resmgf(-\lambda)}<\infty.
\end{equation}

Write $p=e^{-\psi}$ up to a constant.  Then $\psi$ is $C^1$ and convex.
Since an integrable log-concave density tends to zero at $+\infty$,
\begin{equation}
    p(z)=-\int_z^\infty p'(t)\dd t
        =\int_z^\infty\psi'(t)p(t)\dd t.
\end{equation}
Consequently,
\begin{equation}\label{eq:hazard-average-general}
    h(z)=\frac{p(z)}{S(z)}
        =\frac{\int_z^\infty\psi'(t)p(t)\dd t}
               {\int_z^\infty p(t)\dd t}
        \ge\psi'(z),
\end{equation}
where the inequality uses monotonicity of $\psi'$.  Differentiating $h=p/S$
yields
\begin{equation}\label{eq:hazard-derivative-general}
    h'(z)=h(z)\bigl(h(z)-\psi'(z)\bigr)\ge0.
\end{equation}

The inequality is strict.  If $h'(z_0)=0$, then equality holds in
\eqref{eq:hazard-average-general}, and hence
\begin{equation}
    \int_{z_0}^\infty
       \bigl(\psi'(t)-\psi'(z_0)\bigr)p(t)\dd t=0.
\end{equation}
The integrand is continuous and nonnegative, while $p$ is positive.
Therefore $\psi'$ is constant on $[z_0,\infty)$, so $p$ has an exact
exponential right tail.  Integrability forces the exponential rate to be
positive, and such a tail is incompatible with
\eqref{eq:tilted-positive-moments} for all $\theta>0$.  Thus $h'(z)>0$ for
every $z$.  Finally,
\begin{equation}
    (\log G)''(x)=(\log S)''(x-c)=-h'(x-c),
\end{equation}
which proves the lemma.
\end{proof}

\begin{remark}
Log-concavity of $G$ is immediate, since convolution preserves log-concavity. The main point of Lemma~\ref{lem:peeling} is the strict positivity of its logarithmic curvature. The key observation underlying
the proof is that if this curvature vanished at $x$, then the
Esscher-transformed density $p$ would have an exactly exponential right
tail on $[x-c,\infty)$, which is incompatible with the moment assumption
on $T$. 
\end{remark}

We now verify the hypotheses of Lemma~\ref{lem:peeling} for the residual
Airy series.

\begin{proposition}\label{prop:residual}
Let
\begin{equation}\label{eq:T-def}
    T:=\sum_{j=2}^\infty(b_j-X_j).
\end{equation}
Then $T$ has a strictly positive $C^\infty$ log-concave density on $\R$.
Its moment-generating function is finite for every $s>-\lambda_2$.
\end{proposition}

\begin{proof}
As established in Section~\ref{sec:series}, the series converges almost surely and in $L^2$.  Every finite partial sum of the series $\sum_{j=2}^\infty (b_j - X_j)$ has a log-concave density, because a reflected exponential density is log-concave and convolution preserves log-concavity.  Since log-concave probability measures are closed under weak convergence~\citep[Prop. 3.6]{saumard2014log}, the law of $T$ is log-concave. 

Its characteristic function satisfies, for every fixed $N\ge2$,
\begin{equation}\label{eq:T-cf-bound}
    |\E e^{iuT}|
      =\prod_{j=2}^\infty(1+b_j^2u^2)^{-1/2}
      \le\prod_{j=2}^N(1+b_j^2u^2)^{-1/2}.
\end{equation}
Since $N$ is arbitrary, $|u|^m\E e^{iuT}\in L^1(\R)$ for every $m\ge0$, which implies that the density of $T$ is $C^\infty$~\citep[Chap. 8]{foland1999real}.

The support is unbounded below because one exponential summand can be made
arbitrarily large while the independent remainder stays in a bounded set of
positive probability.  It is also unbounded above.  To see this, let 
\begin{equation}
    B_N:=\sum_{j=2}^N b_j,
    \qquad
    U_N:=\sum_{j>N}(b_j-X_j).
\end{equation}
Then $B_N\to\infty$ by \eqref{eq:b-sums}, while $U_N\to0$ in $L^2$.  Given
$A\in\R$, choose $N$ so large that $B_N>A+2$ and
$\Pp(U_N>-1)>0$.  The independent event
$\sum_{j=2}^N X_j<1$ also has positive probability, and on the intersection
one has $T>A$.  The support of a log-concave probability measure is an
interval; hence it is all of $\R$.  A continuous log-concave density is
positive on the interior of its support, so the density is strictly positive
everywhere.

Finally, as shown in Section~\ref{sec:series},  the moment-generating function of $T$ is
\begin{equation}\label{eq:T-mgf}
    \resmgf(s)=\prod_{j=2}^\infty\frac{e^{b_js}}{1+b_js},
\end{equation}
for $s>-\lambda_2$. 
By~\eqref{eq:taylor-Ms} and $\sum_j b_j^2 < \infty$, the product converges and is finite.
\end{proof}

\begin{corollary}\label{cor:v-positive}
The one-sided curvature $\curvg=-(\log g)''$ is strictly positive on $\R$.
\end{corollary}

\begin{proof}
From \eqref{eq:Y-series} and \eqref{eq:T-def},
\begin{equation}
    Y=(d+b_1)-X_1+T.
\end{equation}
Here $X_1\sim\Exp(\lambda_1)$, and $\lambda_1<\lambda_2$.  Proposition
\ref{prop:residual} therefore implies that $\resmgf(s)$ is finite for every
$s\in[-\lambda_1,\infty)$.  Lemma~\ref{lem:peeling}, with
$E=X_1$, $\lambda=\lambda_1$, and $c=d+b_1$, gives
\begin{equation}
    \curvg(x)=-(\log g)''(x)>0
\end{equation}
for every $x\in\R$.
\end{proof}

\begin{remark}
The choice of $X_1$ is forced by the transform domain.  It is the slowest exponential, with the largest mean $b_1$ and smallest rate $\lambda_1$. After removing it, the residual transform is finite at the boundary tilt $-\lambda_1$ because $-\lambda_1>-\lambda_2$.  Removing any later summand would leave a slower exponential in the residual and make the corresponding boundary tilt diverge.
\end{remark}

\section{The Airy convolution identity and the tails}\label{sec:identity}

The second ingredient converts the Airy differential equation into a
one-sided convolution identity.

\begin{lemma}[Airy convolution identity]\label{lem:airy-identity}
Define the nonnegative integrable kernel
\begin{equation}\label{eq:r-def}
    r(u):=2u\sum_{j=1}^\infty e^{-\lambda_j u},
    \qquad u>0.
\end{equation}
Then, for every $x\in\R$,
\begin{equation}\label{eq:airy-convolution-identity}
    \int_0^\infty r(u)g(x+u)\dd u
      =x^2g(x)+\frac12g'(x).
\end{equation}
Consequently, if $g$ is log-concave and
$\curvg=-(\log g)''$, then
\begin{equation}\label{eq:v-ge-4x}
    \curvg(x)\ge4x, \qquad x\in\R.
\end{equation}
\end{lemma}

\begin{proof}
Let
\begin{equation}
    \Psi(s):=\log M(s).
\end{equation}
Differentiating the product \eqref{eq:M-product} twice gives
\begin{equation}\label{eq:Psi-second-product}
    \Psi''(s)=\sum_{j=1}^\infty
       \frac{b_j^2}{(1+b_js)^2}
      =\sum_{j=1}^\infty\frac1{(s+\lambda_j)^2}.
\end{equation}
Note that the differentiation is valid locally uniformly since $\sum_j b_j^2 < \infty$. 

For the function $r$, it is integrable since 
\begin{equation}
    \int_0^\infty  r(u) \dd u = 2 \sum_{j=1}^\infty \int_0^\infty u e^{-\lambda_j u} \dd u = 2 \sum_{j=1}^\infty \frac{1}{\lambda_j^2} < \infty, 
\end{equation}
and its one-sided Laplace transform is
\begin{equation}\label{eq:r-transform}
    \int_0^\infty e^{-su}r(u)\dd u=2\Psi''(s), \qquad \Re(s) > - \lambda_1. 
\end{equation}

In the remainder of the proof, let $s > -\lambda_1$ be real.  
Since $M(s) = \Ai(0) / \Ai(\ascale s)$ by~\eqref{eq:M-def}, 
the Airy equation $\Ai''(z)=z\Ai(z)$ and $\ascale^3=1/2$ give the Riccati identity
\begin{align}
    \Psi''(s)
      &=-\ascale^2\left\{
          \frac{\Ai''(\ascale s)}{\Ai(\ascale s)}
          -\left(\frac{\Ai'(\ascale s)}{\Ai(\ascale s)}\right)^2
        \right\}\notag 
       =(\Psi'(s))^2-\frac{s}{2}.
      \label{eq:airy-riccati}
\end{align}
Since $M''/M=\Psi''+(\Psi')^2$, this is equivalent to
\begin{equation}\label{eq:transform-algebra}
    2\Psi''(s)M(s)=M''(s)-\frac{s}{2}M(s).
\end{equation}

Let
\begin{equation}
    H(x):=\int_0^\infty r(u)g(x+u)\dd u.
\end{equation}
Taking bilateral Laplace transforms, using~\eqref{eq:M-def},~\eqref{eq:r-transform} and~\eqref{eq:transform-algebra}, gives
\begin{equation} \label{eq:laplace-H}
    \int_{\R}e^{sx}H(x)\dd x=2\Psi''(s)M(s) = M''(s)-\frac{s}{2}M(s).
\end{equation}
On the other hand, differentiating $M$ with respect to $s$ yields
\begin{equation}\label{eq:laplace-1}
   M''(s) = \int_{\R} x^2 e^{s x} g(x) \dd x.  
\end{equation}
Further, since the mapping $x \mapsto e^{sx}g(x)$  is log-concave and integrable, it has finite total variation and tends to zero at endpoints. Thus, integration by parts yields 
\begin{align}\label{eq:laplace-2}
    \int_{\R} e^{s x} g'(x) \dd x = - s \int_{\R}   e^{sx} g(x) \dd x = - s M(s).  
\end{align} 
Combining~\eqref{eq:laplace-1} and~\eqref{eq:laplace-2} gives 
\begin{equation}\label{eq:laplace-3}
    \int_{\R}e^{sx}
       \left(x^2g(x)+\frac12g'(x)\right)\dd x
       =M''(s)-\frac{s}{2}M(s).
\end{equation}
Comparing~\eqref{eq:laplace-H} and~\eqref{eq:laplace-3} and using the uniqueness of Laplace transforms proves  \eqref{eq:airy-convolution-identity} almost everywhere. Since both sides of~\eqref{eq:airy-convolution-identity} are continuous, the identity holds everywhere. 

Now define 
\begin{equation}\label{eq:A-def}
    A(x):=\frac{H(x)}{g(x)}
        =\int_0^\infty r(u)\frac{g(x+u)}{g(x)}\dd u.
\end{equation}
For each fixed $u\ge0$, log-concavity of $g$ implies that
$x\mapsto g(x+u)/g(x)$ is nonincreasing. 
Since $r\ge0$, the function $A$ is nonincreasing.  Dividing
\eqref{eq:airy-convolution-identity} by $g(x)$ gives
\begin{equation}\label{eq:A-score}
    A(x)=x^2+\frac12(\log g)'(x).
\end{equation}
The right-hand side is smooth, and therefore
\begin{equation}
    0\ge A'(x)=2x-\frac12 \curvg(x).
\end{equation}
This proves \eqref{eq:v-ge-4x}.
\end{proof}

\begin{remark}[Relation to Menon--Srinivasan]
The key identity~\eqref{eq:airy-convolution-identity} actually already appeared in Menon and Srinivasan~\citep{MenonSrinivasan2010}. They define  $J$ and $\mathcal K$ by
\begin{align}
    \int_{\R}e^{-qz}J(z)\dd z=\frac1{\Ai(q)},\qquad 
    \int_0^\infty e^{-qu}\mathcal K(u)\dd u
      =-2\frac{\dd^2}{\dd q^2}\log\Ai(q),
\end{align}
and use
\begin{equation}
    (\mathcal K*J)(z)
      :=\int_0^\infty\mathcal K(u)J(z-u)\dd u.
\end{equation}
Their identity \cite[Eq.~(125)]{MenonSrinivasan2010} is
\begin{equation}
    z^2J(z)=(\mathcal K*J)(z)+J'(z).
\end{equation}
Under the reflection and rescaling
\begin{equation}
    g(x)=\frac{\Ai(0)}{\ascale}
         J\left(-\frac{x}{\ascale}\right),
    \qquad
    r(u)=\ascale\mathcal K\left(\frac{u}{\ascale}\right),
\end{equation}
this becomes~\eqref{eq:airy-convolution-identity}.
\end{remark}

\section{Proof of Theorem~\ref{thm:main} }

Corollary~\ref{cor:v-positive} proves \eqref{eq:v-positive-main}; in
particular, $g$ is strictly log-concave.  Lemma~\ref{lem:airy-identity} then
proves \eqref{eq:v-linear-main}.  If $x\ge0$, equations
\eqref{eq:v-kappa}, \eqref{eq:v-positive-main}, and
\eqref{eq:v-linear-main} yield
\begin{equation}
    \curv(x)=\curvg(x)+\curvg(-x)>4x.
\end{equation}
The function $\curv$ is even, so
\begin{equation}
    \curv(x)>4|x|, \qquad x\in\R,
\end{equation}
including $x=0$ because $\curv(0)=2\curvg(0)>0$.  This proves
\eqref{eq:kappa-linear-main} and shows that $\curv(x)\to\infty$ as
$|x|\to\infty$.

Finally, $g$ is smooth and strictly positive, so $\curv$ is continuous.
On the compact interval $[-1,1]$, the positive function $\curv$ attains a
strictly positive minimum $m_0$.  Outside that interval,
\eqref{eq:kappa-linear-main} gives $\curv(x)>4$.  Therefore
\begin{equation}
    \inf_{x\in\R}\curv(x)\ge\min\{m_0,4\}>0.
\end{equation}
This proves \eqref{eq:uniform-main} and completes the proof of
Theorem~\ref{thm:main}.

\bibliographystyle{plainnat}
\bibliography{chernoff_strong_logconcavity}

\end{document}